\documentclass{amsart}[12pt] 
\usepackage{latexsym,amssymb,graphicx,amscd,amsmath}

\newtheorem{thm}{Theorem}
\newtheorem{lem}[thm]{Lemma}

\newcommand{\BZ}{{\mathbb{Z}}}
\newcommand{\BQ}{{\mathbb{Q}}}

\newcommand{\BC}{{\mathbb{C}}}

\newcommand{\g}{\gamma}

\DeclareMathOperator{\trace}{Trace}

\begin{document}

\title[On the skein module of the 3-torus]{On the Kauffman bracket skein module of the 3-torus}
 
\author{ Patrick M. Gilmer}
\address{Department of Mathematics\\
Louisiana State University\\
Baton Rouge, LA 70803\\
USA}
\email{gilmer@math.lsu.edu}
\thanks{partially supported by  NSF-DMS-1311911}
\urladdr{www.math.lsu.edu/\textasciitilde gilmer/}

\date{July 11, 2016}

\begin{abstract}  Carrega has shown that the  Kauffman bracket skein module of the 3-torus over  the field of rational function in the variable $A$ can be generated by 9 skein elements. We show this set of generators is linearly independent.
\end{abstract}

\maketitle

Przytycki \cite{P} and Turaev \cite{T1} independently defined the Kauffman bracket skein module. We will consider this skein module over $\BQ(A)$,   the field of rational functions in
a variable $A$. Given a connected oriented 3-manifold $M$, the skein module $K(M)$ is the vector space over $\BQ(A)$ generated by the set of isotopy classes of framed unoriented links in $M$ (including the empty link) modulo the subspace generated by the Kaufman relations:

$$\begin{minipage}{0.4in}\includegraphics[width=0.4in]{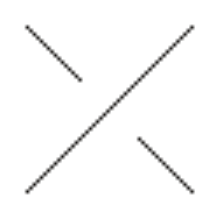}\end{minipage}= A \begin{minipage}{0.4in}\includegraphics[width=0.4in]{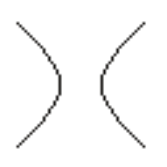}\end{minipage} + A^{-1} \begin{minipage}{0.4in}\includegraphics[width=0.4in]{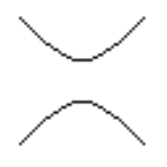}\end{minipage}$$
	$$L \cup \hbox{unknot} = (-A^{2} - A^{-2}) L$$

In the first relation, the  intersection of three links with a 3-ball is $M$ is depicted and  in the exterior of this 3-ball, all 3 links should be completed identically.
In the second relation by an unknot, we mean a loop which bounds a disk in the complement of $L$ with  the framing on the unknot extending to a non-zero normal vector field on the disk.  
Kauffman's argument \cite{K} justifying his bracket approach to the Jones polynomial can be retroactively understood as a proof that $K(S^3)$ is free on the empty link.

It is apparent from the skein definition \cite{Li93,KL,BHMV1,BHMV2} of the Witten-Reshetikhin-Turaev invariant that the WRT invariant of framed links in a fixed closed 3-manifold extends to an invariant of  skein classes in that  3-manifold. This invariant takes the form 
a function on a certain parameter set $\Gamma$ of roots of unity, which are defined by assigning to the variable $A$  an element of $\Gamma$.  Since we are working in this paper with skeins defined over $\BQ(A)$, such a function on  $\Gamma$ may  only be defined for all but a finite number of elements of $\Gamma,$ and two such functions should be considered equal if they agree on all but a finite number of elements of 
$\Gamma.$  For short, we  will use the phrases: ``defined almost everywhere'' and ``agree almost everywhere'' to describe these situations. This idea of using quantum invariants to define invariants of skein classes and show non-triviality or linear          
independence of skein classes appeared in \cite{GH,GZ}. This idea was also used in \cite{G}.

We will work with the TQFT invariants $\left<M,L\right>_{2d+1}$ of a pair consisting of a closed oriented three manifold $M$ (equipped with a $p_1$-structure with $\sigma$-invariant equal to 
 zero) and  an unoriented framed link $L$ \cite{BHMV2}. Here $d$ is an integer greater than or equal to one. We consider the family of these invariants
associated to a choice of $A$ from the set $\Gamma =\{ e^{\pm  \pi i/{(2d+1)} }| d \ge 1 \}$.  This is sometimes called the $SO(3)$ theory. Thus we are working in the theory denoted $V_{2d+1}$ in \cite{ BHMV2}.
This is sometimes called the $SO(3)$ theory. Let $\BC^\Gamma$ 
denote the set of complex valued functions defined almost everywhere on $\Gamma$ where we consider two functions  to be equal if they agree almost everywhere on $\Gamma$. We make  $\BC^\Gamma$ a vector space over $Q(A)$ by setting $R f(\g)= R(\g) f(\g)$ where $R \in Q(A)$, $f \in  \BC^\Gamma$, $\gamma \in \Gamma.$ Taking $\left<M,L\right>$ where $L$ represents a skein class of $K(M)$  defines a linear map
$I:K(M) \rightarrow  \BC^\Gamma$. For a skein $S$ in the 3-torus, we will  use $\left<S\right>$ to denote $\left<S^1 \times S^1 \times S^1,S\right>$ in
$\BC^\Gamma$.

Carrega \cite{C} shows $K(S^1 \times S^1 \times S^1)$  has nine generators. Moreover Carrega shows that these generators will form a basis if the following theorem holds. 

\begin{thm}\label{main}
 The skein given by $1 \times 1 \times S^1$ is non-trivial.
 The skein given by the empty link $\emptyset$ and the skein given by  two parallel copies of  $1 \times 1\times S^1$ form  a linearly independent set in $K(S^1 \times S^1 \times S^1)$.	
\end{thm}

Two parallel copies of a loop is skein equivalent to the loop colored 2 plus the empty link. So  Theorem \ref{main} follows from  Lemmas \ref{one} and \ref{dim} below.

\begin{lem} \label{one}
 We have that $\left<1 \times 1 \times S^1\right> \ne 0 \in \BC^\Gamma.$
In fact,  $\left<1 \times 1 \times S^1\right>$ is the constant function with value one.
  \end{lem}

 \begin{figure}[h]
\centerline{\includegraphics[width=4in]{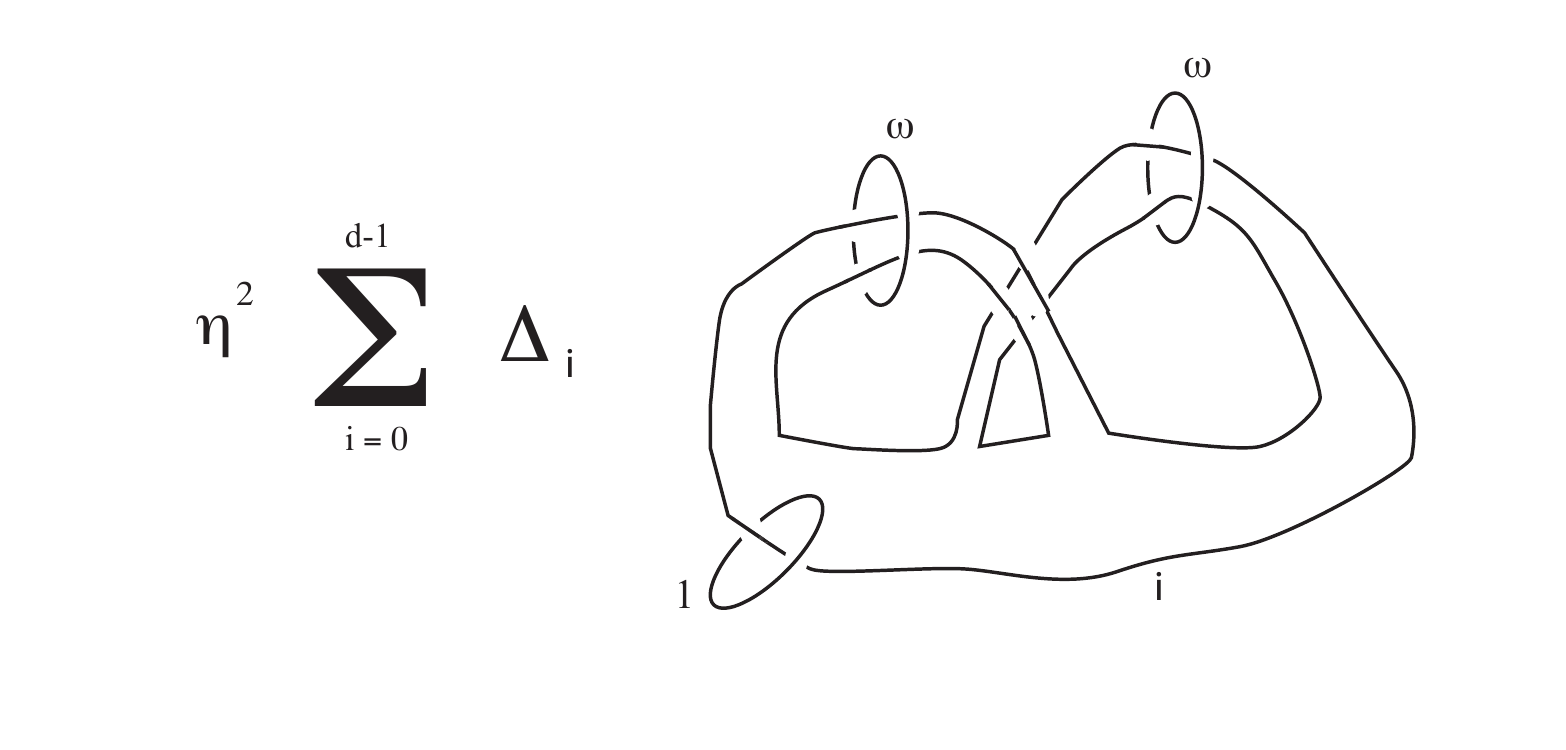}}
\caption{$\left<1 \times 1 \times S^1\right>$}\end{figure}

\begin{proof} The 3-torus can be obtained by doing  surgery on the zero framed Borromean rings. The knot $1 \times 1 \times S^1$ can be taken to be a meridian to one of the components. It follows that $\left<1 \times 1 \times S^1\right>$ is $\eta$ times  the evaluation of  the $\omega$ colored Borromean rings with the meridian of one of them colored one. Here we follow some of the notation of \cite{BHMV2}, with $\omega= \eta \Omega_{2d+1}= \eta \sum_{i=0} ^{d-1} \Delta_i e_i$. 
Here $\Delta_i$ is Kauffman's notation \cite{KL} for the evaluation of an  unknot colored $i$. 
In the diagram, we expand  the $\omega$ colored component with the meridian hanging from it as a sum over $0 \le i \le d-1$ obtaining Figure 1.

We can then do the  simplification given in Figure 2 in two separate locations.
 \begin{figure}[h]\label{lol}
\centerline{\includegraphics[width=4in]{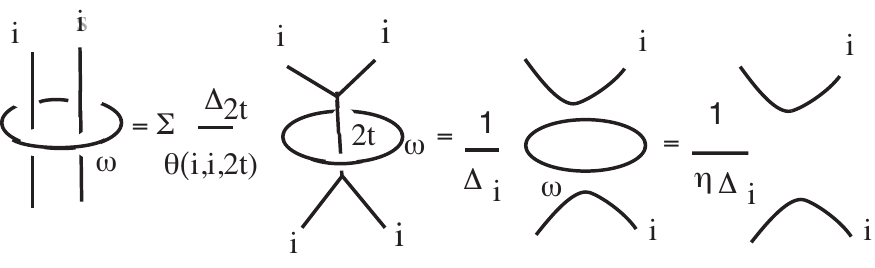}}
\caption{The sum is over $t$  in the range $0 \le t \le i$}\end{figure}
After performing these two simplifications,  we have
that $$\left<1 \times 1 \times S^1\right>= \sum_{i=0} ^{d-1} \frac {H(i,1)}{\Delta_i},$$
where $H(i,a)$ denotes the evaluation of a zero framed hopf link colored with one component colored $i$ and the other  one colored $A$. Using \cite[9.8]{KL}, one sees that $H(i,1)/{\Delta_i}= -A^{2i+2}-A^{-2i-2}$ 
One  can check that $\sum_{i=0} ^{d-1} \frac {H(1,i)}{\Delta_i}=1.$
\end{proof}
  \begin{lem}\label{dim} The two functions  
$\left<1 \times 1 \times S^1(\text{colored $2$})\right>$ and  $\left<\emptyset\right>$ form a linearly independent set in the $Q(A)$ vector space
 $\BC^\Gamma$.
\end{lem}

\begin{lem}\label{dim2}
We have
$$\left<\emptyset\right>=d$$
$$\left<1 \times 1 \times S^1(\text{colored $2$})\right>=d-1$$	
\end{lem}

\begin{proof} By the trace property of TQFTs \cite[1.2]{BHMV2}, we have that $$\left<S^1 \times S^1 \times S^1,\emptyset\right> = 
\trace( Id: V(S^1 \times S^1)\rightarrow V(S^1 \times S^1))=\dim( V (S^1 \times S^1)).$$ 
Similarly: $$\left<1 \times 1 \times S^1(\text{colored $2$})\right>= \dim(V (S^1 \times S^1 \text{with a point colored $2$})).$$  By \cite[4.11]{BHMV2} (or more directly from \cite[p.97]{GM}), for $0 \le c \le d-1$, \newline $\dim(V (S^1 \times S^1 \text{with a point colored  $2c$}))= d-c.$  Thus, $$\dim( V (S^1 \times S^1)=d \quad \text{and} \quad \dim( V (S^1 \times S^1 \text{with a point colored $2$}))=d-1.$$  
\end{proof}
Above, we made use of the trace property of TQFTs that requires as a  hypothesis that the theory satisfy the tensor product axiom. For $V_{2d+1}$, this is true when the functor is restricted to the cobordism category that has as objects surfaces with the property that the sum of the colors of the colored points on each component of a surface  is even.  So we cannot calculate  $\left<1 \times 1 \times S^1 \ \text{(odd colored)} \right>$
 in the same way.  In fact, $\dim (V (S^1 \times S^1 \ \text{with a point colored $1$})=0$, whereas by Lemma \ref{one},  $\left<1 \times 1 \times S^1\right>=1$.

 Nevertheless, we can use trace property of TQFTs  to give a second proof of Lemma \ref{one} as follows:
 By \cite[Lemma 6.3 (iii)]{BHMV1}, one can recolor a link component of a   link $L$ colored one with the color $2d-2$ without changing $\left<L\right>_{2d+1}$. Thus \[\left<1 \times 1 \times S^1 \ \text{(colored 1)}\right>=\left<1 \times 1 \times S^1 \ \text{(colored $2d-2$)}\right> \]  As $\dim (V (S^1 \times S^1  \ \text{with a point colored 2d-2}))=1,$ by the trace formula,  \[\left<1 \times 1 \times S^1 \ \text{(colored $2d-2$)}\right> =1. \]
We remark that one can also give a second proof of  Lemma \ref{dim2} as follows:
By the same argument as given at the beginning of the proof of Lemma \ref{one}, one has that
$$\left<1 \times 1 \times S^1 (\text{(colored a)})\right>= \sum_{i=0} ^{d-1} \frac {H(i,a)}{\Delta_i}.$$ When $a=0$, the right hand side is clearly $d$. One can also show that when $a=2$ the right hand side  is $d-1$  using $H(i,a)=(-1)^{a+i}[(a+1)(i+1)]$ where $[n]= \frac {A^{2n}-A^{-2n}}{A^{2}-A^{-2}}$ ({\em e.g.} \cite[p. 350]{G2}).

\begin{proof}[Proof of Lemma \ref{dim}] 
We assume that there is a $F(A)\in Q(A)$, such that
 $$F( e^{\pm \pi i/(2d+1)}) =\left<1 \times 1 \times S^1(\text{(colored $2$)})\right>/\left<\emptyset\right>=\frac {d-1}{d}$$
for almost all $e^{\pm \pi i/(2d+1)}$'s, and seek a contradiction.
 Taking the limit as $d \to \infty$, we have that $F(1)=1$. We now exhibit an holomorphic function which agrees with $F(A)$ at an infinite set of points which contains a limit point. Let $\log$ be the inverse of $\exp$ defined on $U=\BC \setminus \{ z| \Re(z)\le 0, \Im(z) =0 \}$ taking values with imaginary part in 
 $(- \pi,\pi)$. Consider $f(z)= \frac{ \pi i-3 \log(z)}{ \pi i- \log(z)}$. One may check that $f(e^{\pi i/(2d+1)})=\frac {d-1}{d}$, $f(e^{-\pi i/(2d+1)})=\frac {d+2}{d+1}$
 and $f(1)=1.$ Thus the two holomorphic functions $F$ and $f$ on $U\setminus \{\text{the poles of F}\}$  agree on a set with a limit point, namely at $1$ and at almost all $e^{\pi i/(2d+1)}$. So $F$ and $f$ should agree on $U\setminus \{\text{the poles of F}\}$. But $F$ and $f$ 
 disagree on an infinite set of points of the form $e^{-\pi i/(2d+1)}$. This is the contradiction that we seek.
\end{proof}

Together with Harris, we  calculated the skein module of $\mathcal Q$, the quotient space $S^3$ modulo the  quaternion 8-group  \cite{GH}. The skein module of $\mathcal Q$ is strikingly similar to the skein module of the 3-torus.
Both \cite{GH,C} use a decomposition of $K(M)=\oplus_{x \in H_1(M,\BZ_2)} K(M,x)$ into submodules  
$K(M,x)$ which are defined similarly to $K(M)$ except one considers only links representing the 
$\BZ_2$-homology class  $x$. Then symmetries of $\mathcal Q$ and of the 3-torus show that, in each case, the skein submodules indexed by the non-trivial homology classes are all isomorphic. Different arguments are put forward, in each case, to show that a submodule indexed by a non-trivial $\BZ_2$-homology class has
dimension at most one. The proof that we give here that this dimension is one is similar to the proof used in \cite{GH} to prove this for $\mathcal Q$ . Likewise two different type proofs are used in \cite{GH} and \cite{C} to show that the submodules  indexed by the trivial class in each case have dimension at most two. The two proofs (here and in  \cite{GH}) that this dimension is two  are similar. We note that Harris \cite{H} has also calculated the skein module of $\mathcal Q$ over $\BZ[A,A^{-1}]$.

\end{document}